\documentclass[12pt]{article}
\usepackage{amsmath,amssymb,amsthm,latexsym,euscript,amscd}
\usepackage[english]{babel}
\usepackage[all]{xy}
\newtheorem{Theorem}{Theorem}
\newtheorem{Proposition}[Theorem]{Proposition}
\newtheorem{Corollary}[Theorem]{Corollary}
\newtheorem{Lemma}[Theorem]{Lemma}

\newtheorem{Conjecture}[Theorem]{Conjecture}

\newenvironment{ackn}{\medskip \noindent \small
{\sl Acknowledgments.}}{\bigskip}

\newcommand{\pg}{{\cal P}(\Gamma)}

\newcommand{\Tr}{{\rm Tr}}

\title{
\Large\sc Orthogonal pairs and mutually unbiased bases}

\author{Alexey Bondal\thanks{ Steklov Institute of Mathematics, Moscow, and Kavli Institute for the Physics and Mathematics of the Universe (WPI), The University of Tokyo, Kashiwa, Chiba 277-8583, Japan,  and HSE Laboratory of algebraic geometry, Moscow, and The Institute for Fundamental Science, Moscow}~ and Ilya Zhdanovskiy\thanks{MIPT, Moscow, and HSE Laboratory of algebraic geometry, Moscow}}

\begin{document}

\maketitle

\section{Introduction}
An {\em orthogonal pair} in a semisimple Lie algebra is a pair of Cartan subalgebras which are orthogonal with respect to the Killing form. Description of orthogonal pairs in a given Lie algebra is an important step in the classification of {\em orthogonal decompositions}, i.e. decompositions of the Lie algebra into the sum of Cartan subalgebras pairwise orthogonal with respect to the Killing form.

Orthogonal decompositions  come up firstly in the theory of integer lattices in the paper by Thompson \cite{Tom}. Then the theory of such decompositions was substentially developed \cite{KT}. The classification problem of orthogonal pairs in $sl(n, {\mathbb C})$ is closely related to the classification of complex Hadamard $n\times n$ matrices \cite{KT}, \cite{BTSW}.

Independently, the study in Quantum Theory brought into light the notion of mutually unbiased bases, objects of constant use in Quantum Information Theory, Quantum Tomography, etc. \cite{DEBZ}, \cite{Ru}. It was revealed that mutually unbiased bases are a unitary version of orthogonal pairs \cite{BTSW}. This makes a link of the subject to various vibrant problems in Mathematical Physics.
%such as constructing a firm mathematical background for quantum teleportation of particles with multiple qubits.

One of the reasons why mutually unbiased bases are important in practice is that they provide a crucial mathematical tool that allows to transfer
quantum information with minimal loss of it in the channel. Reliable
protocols in quantum channels are based on a
choice of maximal number of mutually unbiased bases in the relevant vector space of
quantum states of transmitted particles. For instance, protocol BB84, which utilizes 3 such bases in a 2 dimensional
vector space, enables to significantly extend the distance between the source
and the receiver of quantum information. Constructing maximal number of mutually unbiased bases in vector spaces of
higher dimensional is important for producing reliable protocols in quantum channels.

Also, one of the important problems of quantum teleportation is to check the result of
purity of the teleportation by means of Quantum Tomography. This is used in real experiments on teleportation of entangled particles (cf. \cite{LNGGNDVS}). The
Quantum Tomography with minimal error bar is again based on mutually
unbiased bases \cite{BLDG}, \cite{FM}.

Despite of simple definition, the classification of orthogonal pairs is a very hard problem of algebraic geometric flavor. We will consider pairs in Lie algebra $sl(n,{\mathbb C})$. According to the famous Winnie-the-Pooh conjecture \cite{KKU1}, orthogonal decompositions are possible in this algebra when $n$ is a power of prime number only. This suggests the idea that the behavior of the objects under the study strongly depend on the arithmetic properties of the number $n$. For $n=1,2,3$, there is a unique, up to natural symmetries, orthogonal pair. For $n=5$, there are three of them \cite{KKU2}, \cite{No}, while, for $n=4$ (the first non-prime integer), there is a one dimensional family of pairs parameterized by a rational curve.

The first positive integer which is not a power of prime is $n=6$. Winnie-the-Pooh conjecture is open even for this case. Researchers in the quantum information theory have independently come to the unitary version of the Winnie-the-Pooh conjecture, which claims non-existence of $n+1$ mutually unbiased bases in the $n$-dimensional complex space \cite{KKU1} when $n$ is not a power of prime. The case $n=6$ is the subject of problem number 13 in the popular list of problems in Quantum Information Theory \cite{pro}.

In this paper, we outline the proof of existence of a 4-dimensional family of orthogonal pairs in Lie algebra $sl(6,{\mathbb C})$. The existence of such a  family was conjectured by the authors (unpublished). Independently, mathematical physicists came to the conjecture on existence of a 4-dimensional family of pairs of mutually unbiased bases in ${\mathbb C}^6$  \cite{Szo},\cite{MS}. Despite of many efforts the proof of the existence of the family was not available until recently \cite{BZ2}. Our proof is quite involved and requires a lot of algebraic geometry. In this paper, we give a relatively short survey of the main steps of the proof and describe explicit constructions that lead to the existence of the family.

Then, we give an application of the result on the algebraic geometric family of pairs to the study of mutually unbiased bases. We show the existence of a real 4-dimensional dimensional manifold parameterizing pairs of such bases in ${\mathbb C}^6$, thus confirming the conjecture of physicists. The proof is based on construction of a principal homogeneous bundle over the locus ${\cal M}_{\mathbb R}$ parameterizing pairs of mutually unbiased bases.

In \cite{BZ}, we interpreted orthogonal pairs and decompositions as representations of the algebra $B(\Gamma )$ for a suitable choice of graph $\Gamma $ (see section \ref{TLieb}). These algebras are so-called {\em homotopes} over the path algebras of graph $\Gamma$ considered as a topological space. In its turn, path algebras of the graphs are Morita equivalent to the group algebras of fundamental groups of the graphs. This is useful for calculating the moduli space of representations of $B(\Gamma )$.
Orthogonal pairs in $sl(n)$ correspond to representations of algebra $B(\Gamma)$, where $\Gamma$ is the complete bipartite graphs $\Gamma_{n,n}$.
%In the section \ref{algpr},

%In the course of the prove of the main result of the paper, we present various %relevant algebras as free products of two algebras over a third one and explore %these facts for describing the moduli spaces of their representations.

One of the key point of our proof is a hidden geometry of elliptic fibration of a moduli spaces, $X$, of 6 dimensional representations of $B(\Gamma_{3,3})$, where $\Gamma_{3,3}$ is a full bipartite graph of length $(3,3)$. We define 3 functions on $X$ which determine a map $X\to {\cal U}$, where $\cal U$ is a three dimensional affine space. After factorization of $X$ by permutation group $S_3\times S_3$, the fibre is actually isomorphic to (an open affine subset in) two disjoint copies of an elliptic curve. The profit of this map is that the original problem of describing orthogonal pairs in $sl(6,{\mathbb C})$ can be interpreted in terms of 'gluing' four copies of $X$ in such a way that all constructions are basically implemented relatively over ${\cal U}$. The geometry of the elliptic fibration is a powerful tool that eventually allowed us to show the existence of the 4-dimensional family. In particular, we study the interplay of relevant involutions acting on the elliptic fibers. This part is based on heavy use of algebraic geometry. Let us mention important  formula (\ref{identity pq}) which probably needs a more conceptual explanation than just a verification.

If we think about the main steps of the proof in terms of the $6\times 6$ matrix $A$ that conjugates one Cartan subalgebra in the orthogonal pair to the other one ({\em suitable} or {\em generalized Hadamard} matrix), then we first present this matrix in  2 blocks of $3\times 6$ matrices and then decompose each of these $3\times 6$ blocks into two $3\times 3$ blocks.

Equivalently, the first decomposition is about decomposing the set of vertices in one of the rows of the full bipartite graph $\Gamma_{6,6}$ into two disjoint subsets with 3 elements in each. This has a geometric interpretation in the statement of theorem \ref{x66} that the higher dimensional components of moduli space $X(6,6)$ of 6-dimensional representations of algebra $B_{6,6}$, a quotient of algebra $B(\Gamma_{6,6})$, are birationally identified with fiber product of two copies of representation moduli spaces $X(3,6)$ for the algebra $B_{3,6}$, which is a quotient of $B(\Gamma_{3,6})$.

Further, the vertices in the row of length 6 in the full bipartite graph $\Gamma_{3,6}$ are decomposed into two disjoint subsets with 3 elements in each. This boils down to the decomposition of the unique 4-dimensional component of moduli space $X(3,6)$ of representations for $B_{3,6}$ into a fiber product of two copies of moduli $X=X_{3,3}$ for representations of algebra $B(\Gamma_{3,3})$ as in theorem \ref{irrx}. In the text, we do this in the reverse order: first decompose $X=X_{3,6}$ and then $X=X_{6,6}$.

The fiber products are taken over moduli spaces of representations for algebras $A(n)$, $n=3,6$ (see \ref{Morita}).
We construct Morita equivalence of algebra $A(n)$ with the deformed preprojective algebra, for arbitrary $n$. The deformed preprojective algebras are intensively studied by many authors (cf. \cite{GG}, \cite{CB1}). For our purposes, this Morita equivalence is important, because we can use a result of Crawley-Boewey \cite{CB2} to conclude about irreducibility of representation moduli space $Y(n)$ for $A(n)$.
The symplectic geometry of $Y(n)$ is a part of the symplectic approach to the study of pairs of mutually unbiased bases discussed in \cite{BZ1}, where its relation via mirror symmetry to the Birkhoff-Von Neumann polytope of doubly stochastic matrices was discovered.

%This reduces the problem to the study of the fibres of the above map.

We construct an involution on the quotient space $X(3, 6)/S_3$. The crucial step in our argument is to show that this involution agrees with a map $X(3, 6)/S_3\to Y(6)$ and an involution $\sigma '$ on $Y(6)$. The proof of this fact (Proposition \ref{gpsg}) uses the property of automorphisms on varieties of general type to be of finite order.
We use the point $x_0\in X(6,6)$ corresponding to the standard pair of Cartan subalgebras, which has a regular behavior with respect to our constructions, to prove the existence of a 4-dimensional component that contains this point.

Then, we shift our attention to mutually unbiased bases. We compare space ${\cal M}_{\mathbb R}$ parameterizing the pairs of mutually unbiased bases with the space ${\cal M}^{\theta}$ parameterizing stable points of an anti-holomorphic involution $\theta$ acting on the moduli space of orthogonal pairs. We show that ${\cal M}_{\mathbb R}$ is open in ${\cal M}^{\theta}$.
The proof is based on considering a principal homogeneous bundle over ${\cal M}^{\theta}$ and characterizing its restriction to ${\cal M}_{\mathbb R}$ by means of the Sylvester theorem characterizing positive Hermitian matrices. This describes, in principal, the strict polynomial inequalities that define ${\cal M}_{\mathbb R}$ inside ${\cal M}^{\theta}$. Since point $x_0$ is in ${\cal M}_{\mathbb R}$ and real dimension of ${\cal M}_{\mathbb R}$ equals the complex dimension of the corresponding component in $X(6,6)$, we conclude with the existence of a real 4-dimensional family of pairs of mutually unbiased bases.

\bigskip
\begin{ackn}This work was done during authors visit to Kavli IPMU and was supported by World Premier International Research Center Initiative (WPI Initiative), MEXT, Japan. The reported study was partially supported by RFBR, research projects 13-01-00234, ╣14-01-00416 and ╣ 15-51-50045. The article was prepared within the framework of a subsidy granted to the HSE by the Government of the Russian Federation for the implementation of the Global Competitiveness Program.
\end{ackn}

\section{Algebraic preliminaries}\label{algebraic}
\subsection{Orthogonal Cartan subalgebras}\label{cartan}

Consider  a simple Lie algebra $L$ over an algebraically closed field of characteristic zero.
%{\color{blue}$\CC$,
Let $K$ be the Killing form on $L$.
In 1960, J.G.Thompson, in course of constructing integer quadratic lattices with interesting properties, introduced the following definitions.

{\bf Definition.} Two Cartan subalgebras $H_1$ and $H_2$
in $L$ are said to be {\it orthogonal} if $K(h_1,h_2) = 0$ for all $h_1 \in H_1,
h_2 \in H_2$.

{\bf Definition.} Decomposition of $L$ into the direct sum
of Cartan subalgebras $L = \oplus^{h+1}_{i=1}
H_i$ is said to be {\it orthogonal} if
$H_i$ is orthogonal to $H_j$, for all $i \ne j$.

Intensive study of orthogonal decompositions has been undertaken since then (see the book \cite{KT} and references therein).
For Lie algebra $sl(n)$, A.I. Kostrikin et al. arrived to the following conjecture, called {\it Winnie-the-Pooh Conjecture} (cf. {\em ibid.} where, in particular, the name of the conjecture is explained by a wordplay in the
Milne's book in Russian translation).

\begin{Conjecture}
Lie algebra $sl(n)$ has an orthogonal decomposition if and only if $n = p^m$, for a prime number $p$.
\end{Conjecture}

The conjecture has proved to be notoriously difficult. Even the non-existence of an orthogonal decomposition for $sl(6)$, when $n=6$, i.e. the first number which is not a prime power, is still open. Also it is important to find the maximal number of pairwise orthogonal Cartan subalgebras in $sl(n)$ for any given $n$ as well as to classify them up to obvious symmetries.

We recall an interpretation of the problem in terms of systems of minimal projectors and its relation to representation theory of Temperley-Lieb algebras
%and Hecke algebras of some graphs. This was discovered by the first author about 25 years ago (cf. {\it %ibid.}).

Let $sl(V)$ be the Lie algebra of traceless operators in $V$. Killing form is given by the trace of product of operators. A Cartan subalgebra $H$ in $V$ defines a unique maximal set of minimal orthogonal projectors in $V$. Indeed, $H$ can be extended to the Cartan subalgebra $H'$ in $gl(V)$ spanned by $H$ and the identity operator $E$.
Rank 1 projectors in $H'$
%the Cartan subalgebra of $gl(V)$
are pairwise orthogonal and comprise the required set. We say that these projectors are {\it associated} to $H$.

If $p$ is a minimal projector in $H'$, then trace of $p$ is 1, hence, $p-\frac 1nE$ is in $H$. If projectors $p$ and $q$ are associated to orthogonal Cartan subalgebras, then
$$
{\rm Tr}(p-\frac 1nE)(q-\frac 1nE)=0,
$$
which is equivalent to
\begin{equation}\label{algunbias}
{\rm Tr}pq={\frac1n}.
\end{equation}
We say that a pair of minimal projectors is {\em algebraically unbiased} if it satisfies this equation.

Therefore, an orthogonal pair of Cartan subalgebras is in one-to-one correspondence with two maximal sets of minimal orthogonal projectors such that every pair of projectors from different sets is algebraically unbiased . Similarly, orthogonal decompositions of $sl(n)$ correspond to $n+1$ pairwise algebraically unbiased sets of minimal orthogonal projectors. In the analysis of the problem, it is worthwhile to consider not only maximal sets of orthogonal projectors, but also study pairwise unbiasedness for various subsets of maximal sets.
%Thus, we come to the problem of studying the sets of projectors where every pair satisfies either conditions (\ref{aunbias1}-\ref{aunbias2}) or (\ref{orthog}).
This suggests to consider the representation theory of reduced Temperley-Lieb algebras of arbitrary graphs with no loop, which we describe in the next section.

%More explicitly, algebraic unbiasedness can be expressed as follows. Let projectors $p$ and $q$ be given as
%$$
%p=e\otimes x,\ \ q=f\otimes y ,
%$$
%where $e$ and $f$ are in $V$ and $x$ and $y$ are in $V^*$. The equations $p^2=p$ and $q^2=q$ imply:
%\begin{equation}\label{projector}
%(e, x)=1, \ \ (f, y)=1,
%\end{equation}
%where $(-,-)$ stands for the pairing between vectors and covectors.
%Then the algebraic unbiasedness of $p$ and $q$ reads:
%\begin{equation}\label{aunbias}
%(x, f)(y, e)=\frac 1n .
%\end{equation}
%Orthogonality conditions (\ref{orthog}) reads:
%\begin{equation}\label{ortho}
%(x,f)=0, \ \ (y,e)=0.
%\end{equation}

\subsection{Reduced Temperley-Lieb algebras}\label{TLieb}

Let $\Gamma$ be a connected simply laced graph with no loop (i.e. no edge with coinciding ends).  Denote by $V(\Gamma)$ and $E(\Gamma)$ the sets
of vertices and edges of the graph.
%Let $\kk$ be a commutative ring $\kk =F[r, r^{-1}]$, where
Let $F$ be a field of characteristic zero.

%Assign a variable $s_{ij}$ to every edge $(ij)$ of the graph.
Fix $r \in F^*$.
We define
reduced Temperley-Lieb
algebra $B_r(\Gamma)$ as a unital algebra over $F$ with generators $x_i$ numbered by
vertices $i\in V(\Gamma)$. They subject relations:
\begin{itemize}
\item{$x^2_i = x_i$, for every $i$ in $V(\Gamma )$,}

\item{$x_ix_jx_i = rx_i$, $x_jx_ix_j = rx_j$, if there is an edge $(i,j)$ in $\Gamma$,}

\item{$x_ix_j = x_jx_i = 0$, if there is no edge $(i,j)$ in $\Gamma$.}
\end{itemize}

If we replace the last relation by $x_ix_j = x_jx_i$ (under the same condition on $(i,j)$), we get the standard Temperley-Lieb algebra $TL_r(\Gamma)$.
It follows that $B_r(\Gamma)$ is a quotient of Temperley-Lieb algebra $TL_r(\Gamma)$ of graph $\Gamma$. In its turn Temperley-Lieb algebra is a quotient of Hecke algebra of the graph, hence algebra $B_r(\Gamma)$ is a special quotient of Hecke algebra (see \cite{BZ}). Thus representation theory of $B_r(\Gamma)$ is a part of representation theory of Hecke algebras of graphs. Note that the representation theory of $B_r(\Gamma)$ is difficult, and the measure of difficulty is the rank of the first homology of the graph as a topological space.
Clearly, any automorphism of graph $\Gamma$ induces an automorphism of algebra $B_r(\Gamma)$.

The condition on two minimal projectors to be algebraically unbiased (\ref{algunbias}) can be reformulated as  algebraic relations:
$$
pqp=\frac 1n p, \ \ qpq=\frac 1n q .
$$
It follows from section \ref{cartan} that a pair of orthogonal Cartan subalgebras in Lie algebra $sl(n)$ defines a representation of $B_{\frac 1n}(\Gamma_{n,n})$ where $\Gamma_{n,n}$ is a full bipartite graph with $n$ vertices in both rows, and every generator $x_i$ is represented by a rank 1 projector. Generators in one row correspond to the system of orthogonal projectors related to one Cartan subalgebra. Since the sum of all minimal projectors in one system is the identity matrix, the representation descends to a representation of the algebra
$$
B_{n,n}=B_{\frac 1n}(\Gamma_{n,n})/(\sum p_i-1, \sum q_j -1),
$$ where $p_i$'s are idempotents corresponding to one row and $q_j$'s to the other one. Representations of $B_{n,n}$ where every generating idempotent is presented by a minimal projector are in one-to-one correspondence with orthogonal pairs of Cartan subalgebras in $sl(n)$. Moduli of 6-dimensional representations for $B_{6,6}$ is the central object of this paper.

It is instructive to think about $B_r(\Gamma)$ as a {\em homotope} of the path algebra of the quiver (see below).

\subsection{The path algebra of a graph}
Let again $\Gamma$ be a simply-laced graph with no loop. Consider it as a topological space. Let $\pg$ be the Poincare
groupoid of graph $\Gamma$, i.e. a category with objects vertices of the graph
and morphisms homotopic classes of paths. Composition of morphisms is given by concatenation of paths.

%We again fix the set of (commuting) variables $s_{ij}$ related to all edges $(ij)$ of the graph.
Denote by $F\Gamma $ the
algebra over $F $
%$\kk=k[\{ s_{ij}\} ,\{ s_{ij}^{-1}\} ]$
with a free $F$-basis numbered by morphisms in
$\pg$ and multiplication induced by concatenation of paths (when it makes sense, and multiplication is zero when it does not). Let $e_i$ be the element of $F\Gamma$ which is the constant path at vertex $i$. Any oriented edge $(ij)$ can be interpreted as
a morphism in $\pg$, hence it gives an element $l_{ij}$ in $F\Gamma$. These are the generators.
The defining relations are:
\begin{itemize}
\item{$e_i e_j = \delta_{ij} e_i,\ \ e_i l_{jk} = \delta_{ij} l_{ik},\ \ l_{jk} e_i = \delta_{ki} l_{jk}$;}
\item{$l_{ij} l_{ji} = e_i,\ \ l_{ji}l_{ij} = e_j,\ \ l_{ij}l_{km} = 0$, if $j \ne k$.}
\end{itemize}

We consider $F\Gamma$ as an algebra with unit:
$$
1=\sum_{i\in V(\Gamma )}e_i.
$$
%Let $\gamma\in \pg$ be a path in $\Gamma$. Denote by $l_{\gamma}$ the element in $\kk\Gamma$ corresponding to this path. There is %an involutive anti-isomorphism $\sigma :\kk\Gamma\to \kk\Gamma ^{opp}$ defined by

Let $\Gamma$ be in addition a connected graph. Then the category of representations for $F\Gamma$ and
for the fundamental group of the graph are equivalent.
%path algebra ${\mathbb Z} \pg$.
To see this, fix $t \in V(\Gamma)$. Denote by $F[\pi(\Gamma,t)]$
%${\mathbb Z}[\pi(\Gamma,t)]$
the group algebra of the fundamental group $\pi(\Gamma,t)$. Consider
projective $F\Gamma$ - module $P_t = F\Gamma e_t$. Clearly, $P_t$ is a $
F\Gamma$ - $F[\pi(\Gamma,t)]$ - bimodule. Note that $P_t$ are isomorphic as left $F \Gamma$-modules for all choices of vertex $t$. Indeed, the right multiplication by an element corresponding to a path starting at $t_1$ and ending at $t_2$ give an isomorphism $P_{t_1}\simeq P_{t_2}$.

%Using Morita equivalence, we get the following:
Bimodule $P_t$ induces a Morita equivalence between $F\Gamma$ and $F[\pi(\Gamma,t)]$.
Thus, the categories $F\Gamma - {\rm mod}$ and $F[\pi(\Gamma,t)] - {\rm mod}$ are
equivalent. Moreover, algebra $F\Gamma$ is isomorphic to the matrix algebra over $F[\pi(\Gamma,t)]$, with the size of (square) matrices equal to $|V(\Gamma )|$.
%\begin{proof} Straightforward.
%\end{proof}

Mutually inverse functors that induce an equivalence between categories $F\Gamma - {\rm mod}$ and $F[\pi(\Gamma,t)] - {\rm mod}$ are:
\begin{equation}
\label{functors} V \mapsto P_t \otimes_{F[\pi(\Gamma,t)]} V,\ \ W \mapsto
{\rm Hom}_{F\Gamma}(P_t, W).
\end{equation}

%Similar equivalence holds between categories $k\Gamma - {\rm mod}$ and $ \kpig -{\rm mod}$.

%If $k$ is algebraically closed, then

In order to define an isomorphism $F\Gamma\to {\rm Mat}_n(F[\pi(\Gamma,t)])$, fix a system of paths $\{\gamma_i\}$ connecting the vertex $t$ with every vertex $i$. For any element $\pi\in F[\pi(\Gamma,t)]$ consider an element  $\gamma^{-1}_i\pi\gamma_j$ in $F\Gamma$. The homomorphism is defined be the assignment:
$$
\gamma^{-1}_i\pi\gamma_j\mapsto \pi\cdot E_{ij},
$$
where $E_{ij}$ stands for the elementary matrix with the only nontrivial entry $1$ at $(ij)$-th place. This is clearly a well-defined ring isomorphism.

%Let $\kk =k$ be again a field.
The fundamental group $\pi(\Gamma,t)$ is free with the number of generators equal to the rank of the first homology of the graph regarded as a topological space.

%the equivalence implies that homological
%dimension of category $k\Gamma - {\rm mod}$ is 0 if graph is a tree and 1 otherwise.

%The equivalence takes the duality functor (\ref{dualforkg}) for $k\Gamma -{\rm mod}$ into the standard duality $W\mapsto W^*$ for %representations of the group $\pi(\Gamma,t)$.

\subsection{Homotopes and reduced Temperley-Lieb algebras}\label{hmtop}

Recall the definition of homotope. Given a unital algebra $A$ and an element $\Delta \in A$, one can define a new algebra structure on $A$ by the multiplication:
$$
a\circ b=a\Delta b.
$$
The new algebra might not have a unit. For this reason we adjoin a unit to it and denote the new algebra by $B$:
$$
B=F\cdot 1_B\oplus B^+,
$$
where $B^+$ is the two-sided ideal in $B$ which is $A$ as a vector space with the new multiplicaion.
We say that $B$ is the {\em homotope} over $A$ with respect to $\Delta$.

Algebraic properties of homotopes and their general representation theory is available in \cite{BZ}.

Consider again a simply laced graph $\Gamma$ with no loop. Fix $r\in F^*$. The {\em (generalized) Laplace operator of a graph $\Gamma$} is an element $\Delta$ in the algebra $F\Gamma$ of Poincare groupoid of the graph:
\begin{equation}\label{deltaformula}
\Delta =1+ \sqrt{r}\sum l_{ij},
\end{equation}
where the sum is taken over all oriented edges.

Consider algebra $F \Gamma_{\Delta}=F\cdot1\oplus F \Gamma_{\Delta}^+$, the unital homotope over $F\Gamma$ with respect to element $\Delta$. Note that the algebra is independent of the choice of the of the square root of $r$.
Denote by $x_i$'s the elements in $F{\Gamma}_{\Delta}^+$ that correspond to $e_i$'s in $F{\Gamma }$.
The following theorem realizes $B_r(\Gamma)$ as a unital homotope over the Poincare groupoid $F\Gamma$.
\begin{Theorem}\label{bgammaviagr}\cite{BZ}
%Fix $r \in k^*$.
There is a unique isomorphism of algebras and maximal ideals in them:
\begin{equation}
B_r(\Gamma ) \cong {F{\Gamma}_{\Delta}},\ \ \ B^{+}_r(\Gamma ) \cong F{\Gamma}_{\Delta}^+,
\end{equation}
%and hence, an isomorphism of unital algebras and maximal ideals in them:
%\begin{equation}
%B_r(\Gamma) \cong {F{\Gamma}_{\Delta}}
%\end{equation}
that takes $x_i$ into $e_i$.
%The notation for elements $x_i$'s in $B^{+}(\Gamma)$ and $\kk{\Gamma}_{\Delta}$ with respect to this isomorphism.
\end{Theorem}

This theorem allows us to relate moduli spaces of representations of $B_r(\Gamma )$ with the moduli spaces of the path algebra of the graph. Since the latter algebra is Morita equivalent to the fundamental group of the graph, the link to the representation theory of the free group is implied.

\subsection{Algebra A(n) and Morita equivalence}\label{Morita}

Let us define deformed preprojective algebra $\Pi_{\vec{\lambda}}(Q)$ of a free-loop quiver $Q$. Denote by $Q_0$ and $Q_1$ the sets of vertices and arrows of $Q$ respectively.
Let us construct a {\it double} quiver $Q^d$, that is to each arrow $a \in Q_1$ we add an opposite arrow $a^* \in Q^d_1$. Define commutator $c$ as element $\sum_{a \in Q_1}[a,a^*] \in FQ^d$. For the vector $\vec{\lambda} = (\lambda_1,...,\lambda_m) \in F^m, m = |Q_0|$, we define deformed preprojective algebra as follows:
\begin{equation}
\Pi_{\vec{\lambda}}(Q) = FQ^d/\langle c - \sum^k_{i=1}\lambda_i e_i \rangle
\end{equation}
Fix $r_i \in F^*, i = 1,...,n$.
Consider star quiver ${\cal Q}$ with one central vertex and $n$ vertices at the boundary.
The central vertex is connected with every vertex on the boundary by one out-bound arrow.
%Consider the double quiver ${\cal Q}^d$. %related to this quiver with one arrow $a_i$ from the central vertex to every boundary vertex $v_i$ and arrow $a^*_i$ in the opposite direction.
Let vector $\vec{\lambda}$ be $(-r_1,...,-r_n,1), \sum^n_{i=1}r_i = k$, where $k \in {\mathbb N}$, and $-r_i, i = 1,...,n$, corresponds to the vertices on the boundary and $1$ corresponds to the central vertex.

Consider algebra $A(n)$ with generators $P,q_1,...,q_n$ and relations:
\begin{equation}
P^2 = P, q^2_i = q_i, q_iPq_i = r_i q_i, \sum^n_{i=1}q_i = 1.
\end{equation}
\begin{Proposition}
\label{morita}
Algebra A(n) is Morita equivalent to the deformed preprojective algebra $\Pi_{\vec{\lambda}}({\cal Q})$.
\end{Proposition}
Denote by $Y(n)$ the GIT moduli space of $n$-dimensional $A(n)$-representations where $P$ is presented by a projector of rank $k$ and idempotents $q_j$ are represented by projectors of rank 1. The above proposition allows us to apply results of Crawley-Boewey \cite{CB1}, \cite{CB2}. By checking his assumptions for the star quiver, we get that variety $Y(n)$ is irreducible and has dimension $2(n-k-1)(k-1)$.

\subsection{Coproducts of algebras and moduli of representations}
Consider the quotient algebra
$$
B_{k, n}=B_{\frac 1n}(\Gamma_{k, n})/(\sum q_j-1),
$$
where $q_j$'s are idempotents corresponding to the vertices of the row of length $n$ in the bipartite graph $\Gamma_{k, n}$.
A decomposition of the set of vertices in one row of the graph $\Gamma_{n, n}$ into two disjoint subsets with $k$ and $n-k$ elements in each defines two subalgebras $B_{k, n}$ and $B_{n-k, n}$ in algebra $B_{n, n}$. The intersection of this two subalgebras in $B_{n, n}$ is identified with algebra $A(n)$. The importance of algebra $A(n)$ for us is explained by the following proposition.
\begin{Proposition}
Algebra $B_{n,n}$ is a fiber coproduct of $B_{k, n}$ and $B_{n-k, n}$ over $A(n)$.
\end{Proposition}

For algebra $A$, denote by ${\rm Rep}_nA$ the affine variety parameterizing $n$-dimensional representations of $A$.
The above proposition implies:
\begin{Corollary} For every positive $l$, we have the fiber product decomposition:
$$
{\rm Rep}_lB_{n,n}={\rm Rep}_lB_{k,n}\times_{{\rm Rep}_lA(n)}{\rm Rep}_lB_{n-k,n}
$$
\end{Corollary}

Denote by ${\cal M}_nA={\rm Rep}_nA/GL(n)$, the GIT moduli space of $A$-representations. Unfortunately, the fiber coproduct decompositions for algebras does not imply fiber product decompositions for moduli spaces of representations, primarily due to the presence of nontrivial automorphisms of representations.

Denote $X(k, n)={\cal M}_nB_{k,n}$ and $Y(n)={\cal M}_nA(n)$. Consider the open subset $Y(n)_o$ in $Y(n)$ of points corresponding to irreducible representations. Let $X(k, n)_o$ be the open subset in $X(k, n)$ of points corresponding to $B_{k,n}$-representations that restrict to irreducible $A(n)$-representations.

\begin{Proposition}\label{xnn0}
We have:
$$
X(n,n)_o=X(k,n)_o\times_{{Y(n)_o}}X(n-k, n)_o
$$
\end{Proposition}

\section{Moduli spaces of representations for subgraphs of graph $\Gamma_{6,6}$}

\subsection{Representation moduli spaces $X$, $Y$ and $S$}
Let us consider the full bipartite graph $\Gamma_{3,3}$ with 3 vertices in both rows. Denote by $p_i, i=1,2,3$ (respectively, by $q_j, j=1,2,3$) the idempotents in $B_{\frac16}(\Gamma_{3,3})$ corresponding to vertices in the first (respectively, second) row of the graph. Let $X=X_{3,3}$ be the GIT moduli space of 6-dimensional representations for the algebra $B_{\frac16}(\Gamma_{3,3})$ where all idempotents $p_i$ and $q_j$ are presented by projectors of rank 1.

One can check that $X\simeq (F^*)^4$. To this end, one can interpret algebra $B_{\frac16}(\Gamma_{3,3})$ as a {\em homotope} of the path algebra of the graph (see section \ref{algebraic} and \cite{BZ}). Homotope $B$ over an algebra $A$ has a canonical maximal two-sided ideal $B^+$, which is endowed with the left module structure of $A$ that commutes with the right action of $B$ (see subsection \ref{hmtop}). This allows us to consider functor ${\rm Hom}_B(B^+, -):{\rm mod B}\to {\rm mod}A$.

Applying this general theory to $B_{\frac16}(\Gamma_{3,3})$ as a homotope of the path algebra $F\Gamma_{3,3}$ of the graph,
and taking into account the fact that $F\Gamma_{3,3}$ is Morita equivalent to the group algebra of the fundamental group of the graph, which is a free group in 4 generators, implies that the above functor has an interpretation as a functor that takes $B_{\frac16}(\Gamma_{3,3})$-modules to representations of the fundamental group. Moreover, the representations that are parameterized by $X$ are taken to representations of dimension 1. The moduli space of the latter is $(F^*)^4$, hence the map $X\to (F^*)^4$. One can see that the map is one-to-one on closed points, due to interpretation of closed points as equivalence classes of representations. Thus the map is a birational morphism. Since $(F^*)^4$ is smooth, in particular, normal, it follows that the map is an isomorphism.

Algebra $A_3$ has generators $P$ and $q_j, j=1,2,3$, satisfying relations $P^2=P$, $q_j^2=q_j$ and $q_jPq_j=\frac{1}{2}q_j$.
This algebra is endowed with an {\em involution} $\sigma$, which is of particular importance for us. It is given by  $\sigma :P\mapsto 1-P$.
Let $Y$ be the GIT moduli space of 6-dimensional representations of $A_3$ in which $P$ is represented by a projector of rank 3 and $q_j$'s by projectors of rank 1. This is a 4-dimensional variety.

The algebra homomorphism $A_3 \to B_{\frac16}(\Gamma_{3,3})$ given on generators $P\mapsto \sum p_i$ and $q_j\mapsto q_j$ defines a map $f:X_{3,3}\to Y$, which is in fact a quasi-finite map of degree 12.

We will also consider algebra $C$ with generators $P$ and $Q$ and relations $P^2=P$, $Q^2=Q$. The moduli space of 6 dimensional representations for this algebra where both $P$ and $Q$ are represented by projectors of rank 3 and ${\rm Tr}PQ=\frac32$,  is denoted by $S$. It has dimension 2. We have a morphism $g:Y\to S$ defined by the algebra homomorphism $C\to A_3$ that takes $P\mapsto P$ and $Q\mapsto \sum q_j$.

We will consider another copy of $A_3$ with generators denoted by $Q$ and $p_i, i=1,2,3$, which play the roles of $P$ and $q_j, j=1,2,3$, respectively, in the first copy. Then we have a following commutative square of algebras, where we denote algebras together with their generators:

%\begin{equation}\label{baac}
%12345
%\end{equation}

\begin{equation}\label{baac}
\xymatrix{
& B_{\frac16}(\Gamma_{3,3})(p_1,p_2,p_3;q_1,q_2,q_3)\\
A_3(Q;p_1,p_2,p_3)\ar[ru] && A_3(P;q_1,q_2,q_3)\ar[lu]\\
& C(P;Q)\ar[ru]\ar[lu]
}
\end{equation}
In the north-west pointed arrows of this diagram, $P$ is taken to $\sum p_i$ and, in the north-east pointed  arrows, $Q$ goes to $\sum q_j$.
We also have the induced commutative square of moduli spaces:
%\begin{equation}\label{xyyd}
%\xymatrix{& X\ar[rd]^{f} \ar[ld]_{f \circ \tau}&\\ Y \ar[rd]^{} && Y\ar[ld]_{}\\& {\cal D}\ar[d]&}
%\end{equation}

\begin{equation}\label{xyyd}
\xymatrix{
& X\ar[rd]^{f}\ar[ld]_{f \circ \tau}\\
Y\ar[rd] && Y\ar[ld]\\
& S
}
\end{equation}
where $\tau$ is an involution on $X$ which comes from the involution on algebra $B_{\frac16}(\Gamma_{3,3})$ defined by the exchange of $p_i$ with $q_i$, for $i=1,2,3$.

Let us introduce functions $u_1$, $u_2$ on $Y$:
%$\frac{1}{r^2} TrPQPQ$, $\frac{1}{r^3} Tr PQPQPQ$ and $\prod_{(i,j) \in \{1,2,3\}}(\frac{1}{r^2}Tr(Pq_iPq_j)-1)$ respectively.
%Express $u_1$, $u_2$ in terms of $a_{(i,j)}, a_{(i,j,k})$ and $b_{(i,j)}, b_{(i,j,k})$.
%We obtain the following formulas:
\begin{equation}
\label{u1u2u3}
u_1= {6^2}({\rm Tr}Pq_1Pq_2+{\rm Tr}Pq_1Pq_3+{\rm Tr}Pq_2Pq_3),
\end{equation}
\begin{equation}
u_2= {6^3}({\rm Tr}Pq_1Pq_2Pq_3 + {\rm Tr}Pq_1Pq_3Pq_2),
\end{equation}

One can easily check that $u_1$ is ${\rm Tr}PQPQ$ up to a constant multiplier, while $u_2$ can be expressed as a linear combination of ${\rm Tr}PQPQPQ$, ${\rm Tr}PQPQ$ and the unit. It follows that $u_1$ and $u_2$ are well-defined regular functions on $S$, moreover, they generate the algebra of functions $F[S]$.

\subsection{Space $\cal U$}

Now we consider a new function on $Y$:
\begin{equation}
u_3 = ({6^2}\Tr Pq_1Pq_2 - 1)({6^2}\Tr Pq_2Pq_3 - 1)({6^2}\Tr Pq_3Pq_1 - 1).
\end{equation}

We have the 3-dimensional affine space ${\cal U} = {\rm Spec} F[u_1,u_2,u_3]$. It is endowed with natural surjective maps: ${\cal U} \to S$ and $\Theta: Y \to {\cal U}$. The variety ${\cal U}$ is important for us, because many calculations that we perform are done relatively over ${\cal U}$. It would be interesting to find a representation theoretic meaning for ${\cal U}$.

%One can consider elements $u_i$ as elements of $F[x^{\pm 1}_1,x^{\pm 1}_2,y^{\pm 1}_1,y^{\pm 1}_2]$. It can be shown in usual way that $\tau(u_i) = %u_i, i = 1,2,3$.

\begin{Proposition}
\label{propiden}
Consider two systems of orthogonal projectors $(p_1,p_2,p_3)$ and $(q_1,q_2,q_3)$ of rank 1 in a vector space, satisfying condition $\Tr p_i q_j =\frac16 $. Let $P=p_1+p_2+p_3$ and $Q=q_1+q_2+q_3$. Then the following identity holds:
\begin{equation}
\label{identity pq}
\prod_{(i,j) \in \{1,2,3\}}(6^2\Tr(Pq_iPq_j)-1) = \prod_{(i,j) \in \{1,2,3\}}(6^2\Tr(Qp_iQp_j)-1).
\end{equation}
%where product is taken over all non-ordered pairs $(i,j) \in \{1,2,3\}$.
\end{Proposition}

%Also, element $u_3$ is described in proposition \ref{propiden}.

%Consider 3-dimensional affine space ${\cal U} = {\rm Spec} F[u_1,u_2,u_3]$. There exists a natural surjective map: $\U \to {\cal D}$.
%There are natural surjective maps: $\Theta: Y \to {\cal U}$.
%Note that variety $Y_1$ and $Y_2$ are isomorphic by rule: $P \leftrightarrow Q, q_i \leftrightarrow p_i$. Direct checking show that this isomorphism compatible with $\Theta_1$ and $\Theta_2$.

This proposition together with above remarks on $u_1$ and $u_2$, allows us to extend diagram (\ref{xyyd}) to the following commutative diagram:
\begin{equation}
\label{commxyyud}
\xymatrix{& X\ar[rd]^{f} \ar[ld]_{f \circ \tau}\\ Y \ar[rd]^{\Theta}\ar[rdd] && Y\ar[ld]_{\Theta}\ar[ldd]\\& {\cal U}\ar[d] \\ & S}
\end{equation}

The induced map $X \to Y \times_{{\cal U}} Y$ is an embedding. Variety $Y \times_{{\cal U}} Y$ is a divisor in $Y \times_{S} Y$, ${\rm dim} Y \times_{{\cal U}} Y = 5, {\rm dim} Y \times_{S} Y = 6$.

%Using commutative diagram (\ref{commxyyud}), we obtain the following proposition:
%\begin{Proposition}
%$pr_{12}(X) \subset Y \times_{\cal U} Y \subset Y \times_{{\cal D}} Y$.
%\end{Proposition}

Let $S_3$ be the group of permutations in 3 elements. We consider variety $X'=X/(S_3\times S_3)$, where the action of $S_3\times S_3$ on $X$ is induced by the action on $B_{\frac16}(\Gamma_{3,3})$ by independent permutation of $p_i$'s and $q_j$'s. Similarly $Y'=Y/S_3$, where $S_3$ acts on $A_3$ by permuting $q_j$'s, hence the action on $Y$. We have the induced maps $X'\to Y'\to {\cal U}$.

\begin{Proposition}\cite{BZ2}
\label{eliptic}
The fiber of the composite map $X'\to {\cal U}$ over a generic closed point $u\in \cal U$ is a disjoint union of two isomorphic elliptic curves, while the fiber of $Y'\to {\cal U}$ is just one elliptic curve. The map  $X'\to Y'$ maps two components of the fiber of $X'$ over $u$ isomorphically to the fiber of $Y'$ over $u$.
\end{Proposition}

%The generic fiber of $X mod S_3\times S_3$ over ${\cal U}$ is a disjoint union of two elliptic curves.

\subsection{Representation moduli space $X(3,6)$}

Let us consider the full bipartite graph $\Gamma_{3,6}$ with 3 vertices in the first row and 6 vertices in the other one.
Denote by $p_i, i=1,2,3$ (respectively, by $q_j, j=1,\dots , 6$), the idempotents in $B_{\frac16}(\Gamma_{3,6})$ corresponding to vertices in the first (respectively, second) row of the graph. Consider algebra $B_{3,6}$, the quotient of $B_{\frac16}(\Gamma_{3,6})$ by the two-sided ideal generated by $\sum q_j -1$. Let $X(3,6)$ be the GIT moduli space of 6-dimensional representations of algebra $B_{3,6}$ where all idempotents $p_i$ and $q_j$ are represented by projectors of rank 1.

Consider the map $X(3,6)\to X$ induced by the algebra homomorphism $B_{\frac16}(\Gamma_{3,3})\to B_{3,6}$ defined by $p_i\mapsto p_i$ and $q_j\mapsto q_j$. We will also consider a second copy of $B_{\frac16}(\Gamma_{3,3})$ with generators $p_i, i=1,2,3$ and $q_j, j=4,5,6$ and a second map $X_{3,6}\to X$ induced by the similar algebra homomorphism $B_{\frac16}(\Gamma_{3,3})\to B_{3,6}$ defined by $p_i\mapsto p_i$ and $q_j\mapsto q_j$. By combining with two maps $f, \sigma \circ f:X\to Y $, we obtain a commutative diagram:

%\begin{equation}\label{xxxy}
%\xymatrix{
%& X(3,6)\ar[rd]^{p_1}\ar[ld]_{p_2}\\
%X\ar[rd] && X\ar[ld]\\
%& Y
%}
%\end{equation}
\begin{equation}\label{xxxy}
\xymatrix{
& X(3,6)\ar[rd]^{p_1} \ar[ld]_{p_2}\\
X\ar[rd]^{\sigma \circ f} && X\ar[ld]_{f}\\
& Y
}
\end{equation}

\begin{Theorem}
\label{irrx}
Variety $X(3,6)$ is irreducible of dimension 4. Variety $X\times_YX$ has only one irreducible component of dimension 4 and all the other components of lower dimension.
The map $h:X(3,6)\to X\times_YX$ induced by the above diagram establishes a birational isomorphism of $X(3,6)$ with the 4-dimensional irreducible component of $X\times_YX$.
\end{Theorem}

Note that it is quite plausible that $X\times_YX$ is in fact also irreducible, which would mean that map $h$ is birational.

\subsection{Representation moduli spaces $Y(6)$ and $X(6,6)$}

Consider algebra $A(6)$ with generators $P$ and $q_j, j=1,\dots , 6$, and relations:
$$
P^2=P, q_j^2=q_j, q_jPq_j=\frac12q_j, \sum q_j=1.
$$
Algebra $A(6)$ is endowed with the involution $P\mapsto 1-P$ and $q_j\mapsto q_j$.
Denote by $Y(6)$ the GIT moduli space of 6-dimensional representations of algebra $A(6)$ where $P$ is represented by a projector of rank 3 and idempotents $q_j$ are represented by projectors of rank 1. The involution on $A(6)$ induces an involution $\sigma ' :Y(6)\to Y(6)$.

Algebra $A(n)$ is Morita equivalent to the deformed preprojective algebra of the star graph $Q$ with one central vertex and $n$ vertices on the boundary, the central vertex being connected with every boundary vertex by one edge (see section \ref{Morita}). According to Crawley-Boewey result \cite{CB1}, \cite{CB2}, this implies that variety $Y(6)$ is irreducible and has dimension 8.

%Consider the double quiver related to this graph with one arrow $a_i$ from the central vertex to every %boundary vertex $v_i$ and arrow ${\bar a}_i$ in the opposite direction.

%\begin{Proposition}
%Algebra A(n) is Morita equivalent to the preprojective algebra ....
%\end{Proposition}

There is an algebra homomorphism $A(6)\to B_{3,6}$ that takes $P$ to $\sum p_i$. It defines map $g: X(3,6)\to Y(6)$.
Consider the action of group $S_3$ on algebra $B_{3,6}$ which permutes generators $p_1, p_2, p_3$. Clearly, $g$ is an  $S_3$-invariant map. Recall, that according to theorem \ref{irrx} variety $X(3,6)/S_3$ is irreducible.
\begin{Theorem}
\label{birx}
The morphism $g: X(3,6)/S_3\to Y(6)$ maps $X(3,6)/S_3$ birationally on its image in $Y(6)$.
\end{Theorem}

The proof of this theorem heavily uses the fact established in proposition \ref{eliptic} that the fiber of $X/S_3\times S_3$ over a generic point ${\cal U}$ is a disjoint union of two copies of an elliptic curve. This allows to use geometry of elliptic curves and elliptic fibrations.

Consider a second copy of algebra $B_{3,6}$ whose generators we denote by $(p_4,p_5,p_6)$ and $(q_j, j=1,\dots , 6)$. The corresponding moduli space of representations of this algebra is again identified with $X(3, 6)$.

Now consider algebra $B_{6,6}$ which is the quotient of algebra $B_{\frac16}(\Gamma_{6,6})$ with generators $p_i, i=1,\dots , 6$ and $q_j, j=1,\dots , 6$ by the two-sided ideal generated by elements $\sum p_i-1$ and $\sum q_j -1$. Let $X_{6,6}$ be the GIT moduli space of 6-dimensional representations of algebra $B_{6,6}$ where all idempotents $p_i$ and $q_j$ are represented by projectors of rank 1.

Note that the above two copies of algebra $B_{3,6}$ are mapped into algebra $B_{6,6}$ by sending generators $p_i$ to $p_i$ and $q_j$ to $q_j$. We have chosen the indices of the generators in the two copies in such a way that they agreed with the indices of the generators in algebra $B_{6,6}$. These two maps induce two maps $X(6, 6) \to X(3, 6)$. All the above maps can be combined into a commutative diagram:

\begin{equation}
\xymatrix{
& X(6,6)\ar[rd]^{\rm pr_2}\ar[ld]_{\rm pr_1}\\
X(3,6)\ar[rd]^{\sigma'\circ g} && X(3,6)\ar[ld]_{g}\\
& Y(6)
}
\end{equation}
\begin{Lemma}\label{x0}
There exists a point $x_0$ in $X(6,6)$ such that the tangent space $T_{x_0}$ at $x_0$ has dimension 4, the differentials at $x_0$ of maps ${\rm pr}_1$ and ${\rm pr}_2$ are isomorphisms of $T_{x_0}$ with the tangent spaces at the images of $x_0$, and such that the differential of the map $s: X(6,6)\to Y(6)$ induces an embedding of $T_{x_0}$ to the tangent space to $Y(6)$ at $s(x_0)$. The point $s(x_0)\in Y(6)$ corresponds to an irreducible representation of $A(6)$.
\end{Lemma}
\begin{proof}Recall that the {\em standard pair} (see \cite{KKU1}) of Cartan subalgebras in $sl(n, {\mathbb C})$ consists of the diagonal Cartan subalgebra $H_0$ in a fixed basis $\{e_i\}$ and the subalgebra $H_1$ which is linearly spanned by $(P,\dots, P^{n-1})$, where $P$ is the operator of the cyclic permutation of the basis vectors $e_i\mapsto e_{i+1/{\rm mod}n}$.

The transition matrix $A$ from basis $\{e_i\}$ to the basis $\{f_j\}$ related to the second Cartan subalgebra has the following coefficients:
\begin{equation}\label{matrA}
A = \{a_{ij}=\frac{1}{\sqrt {n}}\epsilon ^{(i-1)(j-1)}\},i,j = 1,...,n
\end{equation}
where $\epsilon$ is a primitive root $\epsilon ^n=1$.

%We claim that the standard pair in $sl(6, {\mathbb C})$ is a point in ${\cal M}_{\mathbb R}$ and a smooth point on the %4-dimensional component of $X_{6,6}$.

One can calculate the tangent space to $X_{6,6}$ at the point corresponding to the standard pair and check that it has dimension 4 (cf. \cite{TZ}).

Let us exchange the 3-rd and the 4-th columns of the matrix $A$. This corresponds to reordering of projectors $p_i$'s, thus changing the projections $X(6,6)\to X(3,6)$.
It is a direct check to show that all statements of the lemma are satisfied for this choice of $x_0$ and projections.

%shows that the differentials of the two morphisms $S:X(6,6)\to X(3,6)$ are isomorphisms at $x_0$ and the differential of the %morphsim $X(6,6)\to Y(6)$ at point $x_{0}$ is an embedding \cite{BZ2}. Also one can directly check tht the representation

%Hence it is a smooth point on the component.
%Since $A$ is a unitary matrix, the standard orthogonal pair is an element of ${\cal M}_{\mathbb R}$.
%It assigns to a representation of $B_{6,6}$ a new representation where every generating idempotent of $B_{6,6}$ is presented by the Hermit conjugated projector, i.e. the involution acts on projectors by the rule:

\end{proof}

\begin{Theorem}\label{x66}
The induced morphism $X(6,6)\to X(3,6)\times_{Y(6)}X(3,6)$ establishes a one-to-one correspondence between the set of irreducible components of $X(6,6)$ and $X(3,6)\times_{Y(6)}X(3,6)$ of dimension greater than or equal 4 and birational isomorphisms between corresponding components.
\end{Theorem}
The proof in {\cite{BZ2}} is based on calculation of the locus of points in $X(3,6)\times_{Y(6)}X(3,6)$ which has fiber for $X(6,6)\to X(3,6)\times_{Y(6)}X(3,6)$ different from just one point and showing that it has dimension less than 4.

\section{A 4-dimensional component in $X(6,6)$}
%This section concludes the study of the previous sections with

\subsection{Invariance of the image under an involution}
The main technical result that implies the existence of a 4-dimensional component in $X(6,6)$ is the following statement of independent interest.

\begin{Theorem}\label{x36sigma}
The image of $X(3,6)$ under map $g:X(3,6)\to Y(6)$ has a non-empty Zariski subset which is invariant under involution $\sigma '$.
\end{Theorem}

We describe the main steps of the proof of theorem \ref{x36sigma}.

According to theorem \ref{irrx} variety $X(3, 6)$ is irreducible and is embedded birationally onto the only 4-dimensional irreducible component of $X\times_YX$. Consider the map $h:X\times_YX\to Y\times_SY$.
\begin{Proposition}\label{xyxsigma}
The image under $h$ of the 4-dimensional irreducible component of $X\times_YX$ has a non-empty Zariski open subset which is invariant under involution $(\sigma , \sigma )$.
\end{Proposition}
The map $X(3,6)\to Y\times_SY$ factors through the quotient map $X(3,6)\to X(3, 6)/S_3$, where the action of $S_3$ on $X(3,6)$ is induced by permutations of $p_i, i=1,2,3$.
\begin{Proposition}\label{zariskyiso}The induced morphism $X(3, 6)/S_3\to Y\times_SY$ isomorphically maps a Zariski open subset in $X(3, 6)/S_3$ into $Y\times_SY$.
\end{Proposition}
Propositions \ref{xyxsigma} and \ref{zariskyiso} imply that involution $(\sigma, \sigma )$ induces an involution $\pi$ on a Zariski open subset of $X(3, 6)/S_3$.

Map $g$ allows factorization through the quotient $X(3, 6)\to X(3, 6)/S_3$, thus inducing a map $g: X(3, 6)/S_3 \to Y(6)$.

\begin{Proposition}\label{gpsg}
$g\pi = \sigma 'g$.
\end{Proposition}
\begin{proof}
%The proof of this proposition is divided into two major steps. The first one is
First, wet prove that the involution $\pi$ commutes with the action of $S_6$ on $X(3, 6)/S_3$ that is induced by the permutations of $q_j, j=1,\dots , 6$, in algebra $B(3, 6)$.
Consider the product $Y\times_S Y$ which is defined by the two maps $Y\to S$ that are induced by the maps $C\to A_3$ defined by $Q\mapsto q_1+q_2+q_3$ and by $Q\mapsto 1-q_1-q_2-q_3$.
%The second step of the proof starts from considering

Let us construct a morphism $Y(6)\to Y\times_S Y$. It corresponds to a decomposition of the set $(1,2,3,4,5,6)$ into a disjoint union of two subsets by 3 elements in each and a choice of ordering of elements in each subset. We can assign two algebra homomorphisms $A_3\to A(6)$ to this combinatorial data. The first map takes idempotents $q_j$'s of $A_3$ to $q_j$'s with indices in the first subset, ordered in the prescribed way, and similarly for the second homomorphism. Together, these homomorphisms define a morphism $Y(6)\to Y\times Y$, which is easily seen to  descend to a morphism $Y(6)\to Y\times_S Y$. When composed with $g$, this morphism gives us a morphism $X(3, 6)/S_3\to Y\times_S Y$.

We choose two particular decompositions of the set $(1,2,3,4,5,6)$ into a disjoint union of two subsets. One is $((1,2,3), (4,5,6))$ and the other one is $((1, 2,4), (3,5,6))$. As above they define us two morphisms $X(3, 6)/S_3\to Y\times_S Y$. Let us consider two functions on the variety $X(3,6)/S_3$:
$$
z_1={\rm Tr}Pq_1Pq_2,\ \ z_2={\rm Tr}Pq_5Pq_6.
$$
Let ${\cal Z}={\rm Spec}F[z_1, z_2]$. The natural morphism $X(3,6)/S_3\to {\cal Z}$ factors through both morphisms $X(3, 6)/S_3\to Y\times_S Y$. Hence we get a commutative diagram:

\begin{equation}
\xymatrix{
& X(3,6)/S_3\ar[rd]\ar[ld]\\
Y \times_S Y \ar[rd] && Y \times_S Y\ar[ld] \\
& {\cal Z}
}
\end{equation}

Involution $(\sigma , \sigma )$ acts along the fibers of both morphisms $Y\times_S Y\to {\cal Z}$. Denote by $\pi$ and $\pi '$ the involutions on $X(3, 6)/S_3$, where $\pi$ was defined above, and it is attached to one of the morphisms $X(3, 6)/S_3\to Y\times_SY$, while $\pi '$ is similarly attached to the other morphism $X(3, 6)/S_3\to Y\times_SY$. Both $\pi$ and $\pi '$ act along the fiber of the map
$X(3, 6)/S_3\to {\cal Z}$. Therefore, the product $\pi \pi '$ also acts along the fiber of the same map. The fibers of the map over a generic point are compactified to a surface of general type. There $\pi \pi '$ is a birational automorphisms of the surface of general type. The group of birational automorphisms of the variety of general type is finite (cf. \cite{HMX}). Therefore, element $\pi \pi '$ is of finite order. One can find a smooth fixed point of $\pi \pi '$ on $X(3, 6)/S_3$ such that $\pi \pi '$ acts by identity on the tangent space at this point. The point is a projection to $X(3,6)/S_3$ of the point in $X(6,6)$ corresponding to the 'standard orthogonal pair' of Cartan subalgebras in $sl(6, {\bar F})$.  Since $\pi \pi '$ is of finite order it follows that it is identity on the whole $X(3, 6)/S_3$. Therefore, $\pi = \pi '$.

This implies that $\pi$ commutes with transposition $(34)\in S_6$. Clearly, $\pi$ commutes with all elements in $S_6$ which permute inside the subsets $(1,2,3)$ and $(4,5,6)$. Together with transposition $(34)$ they generate the whole group $S_6$. Thus $\pi$ commutes with it.

Now we consider the product of as many copies of $Y\times_S Y$ as there exist decompositions of set $(1,2,3,4,5,6)$ into a disjoint union of two subsets by 3 elements in each and a choice of ordering of elements in each subset.
%Every such combinatorial decomposition yields a map $Y(6)\to Y\times_S Y$ which is assigned by distributing the projectors $q_j$ %in algebras $A_3$ corresponding to the two copies of $Y$ according to the decomposition.
Taking the product of above maps for each individual copy of $Y\times_S Y$ defines a morphism $\psi :Y(6)\to \prod (Y\times_S Y)$. One can  check that this map is birationally an embedding.

Variety $\prod (Y\times_S Y)$ has an involution $\sigma ''$ defined by the action of $(\sigma , \sigma )$ on every component $Y\times_S Y$. It is obvious from the definition that $\sigma '' \psi = \psi \sigma '$. Denote $\phi = \psi g: X(3, 6)/ S_3\to \prod (Y\times_S Y)$. Since $\pi $ commutes with the action of $S_6$, it follows that $\sigma ''\phi =\phi \pi$.

As $g$ and $\phi$ are both birationally embeddings, it follows that $g\pi = \sigma 'g$.

\end{proof}
It would be nice to have a more conceptual proof for this statement.

Clearly, proposition \ref{gpsg} implies the proof of theorem \ref{x36sigma}.

\subsection{The main algebraic geometric result}

\begin{Theorem}\label{compon}
There exists a  4-dimensional irreducible component of $X(6,6)$ which contains the point $x_0$ constructed in Lemma \ref{x0}.
\end{Theorem}
\begin{proof}
Proposition \ref{gpsg} implies that the variety ${\bar T}$ which is a locus of points $({\bar x}, \pi {\bar x})$, where ${\bar x}$ runs over the set of points $X(3,6)/S_3$ such that $\pi {\bar x}$ is well defined, is a subvariety in $X(3,6)/S_3\times_{Y(6)}X(3,6)/S_3$. Let $T $ be its pre-image in $X(3,6)\times_{Y(6)}X(3,6)$. Consider the open subset $T_o\subset T$ of points which lie over the locus $Y_o$ of irreducible representations for algebra $A(6)$. According to Proposition \ref{xnn0}, the open subset $X(6,6)_o$ is isomorphic to $X(3,6)_o\times_{Y(6)_o}X(3,6)_o$. Thus $T_o$ is a subvariety in $X(6,6)_o$. Note that ${\bar T}$ is irreducible by construction, and $T$ might have several components. By construction,  ${\bar T}$ and all components of $T$ have dimension 4.

Now consider the point $x_0\in X(6,6)$ which was constructed in Lemma \ref{x0}. By the lemma, $x_0$ lies over $Y_o$, i.e. it corresponds to a point in $T_o$ under isomorphism in Proposition \ref{xnn0}. Since the tangent space to $X(6,6)$ at this point is 4 and $T_o$ is of dimension 4, it follows that $x_0$ is a smooth point on $T_o$. Hence the irreducible component of $T_o$ that contains $x_0$ is an irreducible component of $X(6,6)$.

%Let $x_1$ be its image in $X(3, 6)$ under one of the maps $X(6, 6)\to X(3, 6)$ and ${\bar x_{1}}$ its image in $X(3, 6)/S_3$.
%Variety $X(3, 6)$ is irreducible of dimension 4. Since the differentials for the both morphism ${\rm pr}_1, {\rm pr}_2: X(6,6)\to %X(3,6)$ are isomorphisms at point $x_0$, the images ${\rm pr}_1(x_0)$ and ${\rm pr}_2(x_0)$ are smooth points on $X(3,6)$.  Its %image inhas 4-dimensional t

%Note that morphism $g:X(3,6)\to Y(6)$ is quasi-finite, according to theorem  {\ref{birx}}.
%By replacing $X(3,6)$ by a Zariski open subset, we can make $g$ to be finite over its image. According to theorem \ref{x36sigma} %the image is birationally preserved by involution $\sigma '$. At the price of reducing the image by even smaller Zariski open %subset $V$, we can assume that $\sigma '$ is regular. Then the fiber product $X(3,6)\times_{Y(6)}X(3,6)$, when restricted to $V$, %clearly has at least one component of dimension 4.
%Since morphism $X(6,6)\to X(3,6)\times_{Y(6)}X(3,6)$ is birational on irreducible components of dimension 4 according to theorem %\ref{x66}, it follows that $X(6,6)$ has an irreducible component of dimension 4.

\end{proof}

Since $X(6,6)$ can be interpreted as the moduli space of orthogonal pairs in $sl(6)$, as it was explained in section \ref{TLieb}, we have the following result.

\begin{Corollary}
There exists a 4 dimensional family of orthogonal pairs in $sl(6)$, which contains the standard pair.
\end{Corollary}

%Since $X(6,6)$ is the moduli space of orthogonal pairs in $sl(6)$ as explained in section %\ref{TLieb}, the theorem implies existence of 4-dimensional family of such pairs.

It might be instructive to reformulate Proposition \ref{gpsg} in terms of elementary Linear Algebra.

\begin{Proposition}
Let $\cal W$ be the irreducible variety parameterizing $6\times 6$-matrices $P$ of rank 3 with $\frac 12$'s on the diagonal which satisfy $P^2=P$ and admit  a decomposition into three matrices $p_i$ of rank 1 with $\frac 16$ on the diagonal (which implies $p_i^2=p_i$):
$$
P=p_1+p_2+p_3.
$$
Then, for almost all $P\in {\cal W}$, matrix $1-P$ is also in ${\cal W}$.
\end{Proposition}
Chances are that this statement is true for all $P\in {\cal W}$.

\section{Mutually unbiased bases}

\subsection{Mutually unbiased bases and system of projectors}
The terminology of unbiased bases first appeared in physics.
%{\bf
%{\large {\color{red} Protocols of Quantum Information Transmission}}}

Let $V$ be an $n$ dimensional complex vector space with a fixed
Hermitian metric $\langle \ ,\ \rangle $. Two orthonormal Hermitian
bases $\{e_i\}$ and $\{f_j\}$ in $V$ are {\em mutually
unbiased} if, for all $(i,j)$,
\begin{equation}\label{ef}
|\langle e_i, f_j \rangle|^2 =\frac 1n.
\end{equation}

There are two types of obvious transformations acting on the set of mutually unbiased bases. First, one can independently change the phase of all vectors in both bases:
$$
e_j\mapsto {\rm exp}(\sqrt{-1}\alpha_j )e_j,
$$
$$
f_j\mapsto {\rm exp}(\sqrt{-1}\beta_j )f_j.
$$
Second, on can transform all bases by a simultaneous linear transformation from $GL(n, {\mathbb C})$.

Let $\{p_i\}$ be the orthogonal (i.e. $p_ip_j=0$, for $i\ne j$) system of minimal projectors in $V$ related to base $\{e_i\}$, and $\{q_j\}$ the system of minimal projectors related to base $\{f_j\}$. Since both bases are orthonormal, all projectors are Hermitian, i.e. satisfy $p_j^{\dag}=p_j$ and $q_j^{\dag}=q_j$. Moreover, the condition that the bases are mutually unbiased is equivalent to:
$$
{\rm Tr}p_iq_j=\frac 1n,
$$
for all $(i,j)$. The converse is also true: two orthogonal systems of Hermitian projectors satisfying the above equation uniquely define a mutually unbiased pair of bases up to the first type of transformations, i.e. up to changing the phases of basic vectors.

It follows from section \ref{cartan} that a pair of mutually unbiased bases defines a pair of orthogonal Cartan subalgebras in Lie algebra $sl(n,{\mathbb C})$. The requirement that projectors are Hermitian means that the pair of Cartan subalgebras is special. We will see in the next subsection that they parameterize a real submanifold in the moduli space of all pairs of Cartan subalgebras.

\subsection{Moduli of mutually unbiased bases as a 'positive' real form of moduli of orthogonal pairs}
Let ${\bar {\cal X}}$ be the (singular) algebraic variety over $\mathbb C$ that parameterizes all pairs of orthogonal Cartan subalgebras in Lie algebra $sl(V)$, $V\simeq {\mathbb  C}^n$.
Since this is identified with the variety ${\rm Rep}_nB_{n,n}$, it is an affine variety.
Group $GL(V)$ acts on ${\bar {\cal X}}$, and the GIT quotient ${\bar {\cal M}}={\bar {\cal X}}/GL(V)$ is the moduli space of orthogonal pairs in $V$. As this is a GIT factor of an affine variety, it is affine too.

As we know, an orthogonal pair is uniquely defined by a pair of orthogonal systems of minimal projectors, where any pair of projectors from different systems are algebraically unbiased. For brevity, we will call such a pair of systems of projectors by {\em configuration}. A configuration is defined by an $n$-dimensional representation of algebra $B_{n,n}$, which is known to be always irreducible (cf. \cite{Iva}).

We reduce ${\bar {\cal X}}$ to its open subvariety ${\cal X}$ of smooth points, and we denote ${\cal M}={\cal X}/GL(V)$. Let us consider the real subvariety ${\cal X}_{\mathbb R}$ in $\cal X$ which is the locus of points that correspond to algebraically unbiased pair of orthogonal systems of Hermitian projectors. The unitary group $U(n)$ acts on ${\cal X}_{\mathbb R}$ and the quotient ${\cal M}_{\mathbb R}={\cal X}_{\mathbb R}/U(n)$ is the moduli of mutually unbiased bases.

Consider the involution that acts on ${\cal X}$ by Hermitian conjugation of all projectors:
$$
p\mapsto p^{\dag}.
$$
Clearly the involution is anti-holomorphic, and ${\cal X}_{\mathbb R}$ is the locus of stable points of the involution.
It is easy to check that the involution descends to an involution $\theta$ on $\cal M$ and that ${\cal M}_{\mathbb R}$ is embedded into the stable locus ${\cal M}^{\theta}$ of the involution on $\cal M$ . We will show that ${\cal M}_{\mathbb R}$ is an open subset in ${\cal M}^{\theta}$.

%Let $\cal M^o$ be the locus of points that correspond to irreducible representations of $B_(n,n)$.

Let $\mathbb H$ be the set of hermitian operators in $V$, and ${\mathbb H}^{\times}$ be the open subset of invertible Hermitian operators.
Define ${\cal Y}\subset {{\mathbb H}^{\times}}\times {\cal X}$ by
$$
{\cal Y}=\{(g, \{p_i, q_j\})\in {{\mathbb H}^{\times}}\times {\cal X}|\ \  p_i^{\dag}=g^{-1}p_ig, q_i^{\dag}=g^{-1}q_ig\}
$$
Let ${\mathbb H}^{\times}_{\pm}\subset {\mathbb H}^{\times}$ be the open subset of invertible Hermitian matrices which are either positive or negative.
Define ${\cal Y}_{\pm}\subset {\cal Y}$ the open subset of those $(g, \{p_i, q_j\})$ for which $g\in {\mathbb H}^{\times}_{\pm}$.

We consider the map $\phi :{\cal Y} \to \cal X$ given by the projection to the second component of ${\mathbb H}^{\times}\times \cal X$ and similar map $\phi_{\pm} :{\cal Y}_{\pm} \to \cal X$.

Denote by ${\mathbb R}^{\times}$ the group of non-zero real numbers.
Consider group $G={\mathbb R}^{\times}\times PGL(n, {\mathbb C})$ and its action on ${{\mathbb H}^{\times}}\times {\cal X}$ by:
$$
(\alpha , h)(g, \{p_i,q_j\})=(\alpha hgh^{\dag}, \{hp_ih^{-1}, hq_jh^{-1}\}).
$$
It is easy to check that ${\cal Y}$ and ${\cal Y}_{\pm}$ are preserved by this action.

\begin{Proposition} ${\cal Y}$ is a principal homogeneous $G$-bundle over ${\cal M}^{\theta}$.
Similarly, ${\cal Y}_{\pm}$ is a principal homogeneous $G$-bundle over ${\cal M}_{\mathbb R}$.
%The action of $G$ on ${\cal Y}$ is free. The quotient ${\cal Y}/G$ is identified with ${\cal %M}^{\theta}$.
%Under this identification, we have: ${\cal Y}_{\pm}/ G={\cal M}_{\mathbb R}$.
\end{Proposition}
\begin{proof} Let us check that the orbits of the action by ${\mathbb R}^{\times}$ are fibers of the map ${\cal Y}\to {\cal X}$. If $(g_1, \{p_i, q_j\})$ and $(g_2, \{p_i, q_j\})$ are in the fiber of ${\cal Y}\to {\cal X}$, then $(g_1)^{-1}g_2$ lies in the stabilizers of all projectors in configuration. Since we consider irreducible representations of $B_{n,n}$, we have by Schur lemma: $(g_1)^{-1}g_2=\lambda \cdot 1$. Therefore,
$$
g_2=\lambda g_1,
$$
where $\lambda \ne 0$, as $g_2$ is invertible.
Since $g_1$ and $g_2$ are Hermitian, applying the Hermitian conjugation gives:
$$
g_2={\bar \lambda }g_1.
$$
Hence, $\lambda ={\bar \lambda }$, i.e. $\lambda \in {\mathbb R}^{\times}$.

As it was already mentioned, any configuration is given by an irreducible representation of $B_{n,n}$.
Therefore, the action of $PGL(n, {\mathbb C})$ on ${\cal X}$ is free, because the stabilizer of any configuration is a scalar matrix by Schur lemma. It follows, that the action of $G$ on $\cal Y$ is free.

%The image $\phi ({\cal Y})$ is $\pi^{-1}(\cal M^{\sigma})$,
%the image $\phi ({\cal Y}_{\pm})$ is $\pi^{-1}(\cal M_{\mathbb R})$
Take a point $m\in \cal M^{\theta}$. A point in $\cal X$ over it is presented by a configuration of projectors $ \{p_i, q_j\}$. Since $m$ is stable under involution $\sigma$ on the quotient space $\cal M$, there exists $g\in GL(n, \mathbb C)$ such that
$$
p^{\dag}=g^{-1}pg,
$$
for every projector $p$ from the configuration.
If we conjugate this equation, we get:
$$
p=g^{\dag}p^{\dag}(g^{\dag})^{-1}.
$$
Together, these equations implies that $g^{\dag}g^{-1}$ stabilizes all projectors $p$ involved.  It follows from Schur lemma that $g^{\dag}g^{-1}=\lambda \cdot 1$, for some nonzero multiplier $\lambda\in {\mathbb C}$.
Hence
$$
g^{\dag}=\lambda g.
$$
By taking Hermitian dual, we have:
$$
{\bar \lambda} g^{\dag}=g,
$$
which, when combined with the previous relation, implies:
$$
|\lambda |^2=1.
$$
It easy to see that we can replace $g$ by $\alpha g$, for some $\alpha \in {\mathbb C}$, and get $g^{\dag }=g$. The inverse inclusion $\phi ({\cal Y})\subset \pi^{-1}(\cal M^{\theta})$ is obvious. This proves that
${\cal Y}/G={\cal M}^{\theta}$.

Now let us check that $\phi ({\cal Y}_{\pm})\subset \pi^{-1}(\cal M_{\mathbb R})$. Take a point $(g, \{p_i, q_j\})\in {\cal Y}_{\pm}$. We can assume that $g>0$, because changing the sign of $G$ does not change the conjugation by it. For positive non-degenerated $g$, it is known to exist a decomposition:
$$
g=v^{\dag}v,
$$
for some invertible operator $v$. Since, for all projectors $p$ in the configuration, we have:
$$
p^{\dag }=g^{-1}pg=v^{-1}(v^{\dag})^{-1}pv^{\dag}v,
$$
it follows that $(v^{\dag})^{-1}pv^{\dag}$ is self-adjoint. Hence, we can conjugate our configuration to a self-adjoint one.

Conversely, take a point $m\in \cal M_{\mathbb R}$. By definition, there exists a point in the $\pi$-fiber of it such that all projectors from its configuration are Hermitian. Let us take another point in the same fiber. Then every projector $p$ from its configuration is conjugate to the corresponding Hermitian projector $r$:
$$
p=h^{-1}rh,
$$
where $h\in GL(n, {\mathbb C})$ is the same for all projectors $p$ of the configuration. Since $r^{\dag}=r$, we have
$$
p^{\dag}=h^{\dag}r^{\dag}(h^{\dag})^{-1}=h^{\dag}r(h^{\dag})^{-1}=h^{\dag}hph^{-1}(h^{\dag})^{-1}.
$$
Since $h^{\dag}h$ is positive, we have $\pi^{-1}(\cal M_{\mathbb R})\subset \phi ({\cal Y}_{\pm})$.
%A point $\cal X$ over it, is represented by
\end{proof}

%\begin{Lemma}
%Fiber of $\phi :\cal Y \to \cal X$ is ???
%\end{Lemma
\begin{Corollary}\label{minequalities}

Subset ${\cal M}_{\mathbb R}\subset {\cal M}^{\theta}$ is open and is defined by a system of strict real polynomial inequalities.
\end{Corollary}
\begin{proof} According to Sylvester theorem, positive Hermitian matrices are given by a system of $n$ strict polynomial inequalities with real (even integer!) coefficients. Hence the open subset ${\cal Y}_{\pm}\subset {\cal Y}$ is defined by strict polynomial inequalities too. Since ${\cal Y}_{\pm}$ is invariant with respect to the free $G$ action, the inequalities descend to strict polynomial inequalities on ${\cal M}^{\theta}$.
\end{proof}

\subsection{A 4-dimensional family of mutually unbiased bases}

Theorem \ref{compon} together with Corollary \ref{minequalities} imply the existence of a 4 dimensional family of mutually unbiased bases in 6-dimensional complex space.

\begin{Theorem}
There exists a family of real dimension 4 of mutually unbiased bases in ${\mathbb C}^6$.
\end{Theorem}
\begin{proof}
%There
We have an anti-holomorphic involution $\theta$ on the moduli space $X(6,6)$ of 6-dimensional representations of $B_{n,n}$. Let us restrict to the locus $\cal M$ of smooth points in all irreducible  components of $X(6,6)$ as above. The locus of stable points of the involution on each component is a smooth real submanifold of real dimension equal to the complex dimension of the component.
By theorem \ref{compon}, we have a 4-dimensional irreducible component in $X(6,6)$. Hence, we need simply to check that the stable locus of $\theta$ is not empty on the smooth part of the component.

%Recall that the {\em standard pair} of Cartan subalgebras in $sl(n, {\mathbb C})$ consists of the diagonal Cartan subalgebra $H_0$ %in the basis $\{e_i\}$ and the subalgebra $H_1$ linearly spanned by $(P,\dots, P^{n-1})$, where $P$ is the operator of the cyclic %permutation of the basis vectors $e_i\mapsto e_{i+1/{\rm mod}n}$.

%The transition matrix $A$ from basis $\{e_i\}$ to the basis $\{f_j\}$ related to the second Cartan subalgebra has coefficients %$\{a_{ij}=\frac{1}{\sqrt {n}}\epsilon ^{(i-1)(j-1)}\}$, where$\epsilon$ is a primitive root $\epsilon ^n=1$.

%We claim that the standard pair in $sl(6, {\mathbb C})$ is a point in ${\cal M}_{\mathbb R}$ and a smooth point on the %4-dimensional component of $X_{6,6}$.

%The birational map $X(6,6)\to X(3,6)\times_{Y(6)}X(3, 6)$ in theorem \ref{x66} depends on decomposition of the set of 6 columns of %matrix $A$ into two disjoint subsets by 3 elements in each. We can choose the subsets $(1,2, 4)$ and $(3, 5, 6)$. One can check %that then the map $X(6,6)\to X(3,6)\times_{Y(6)}X(3, 6)$ takes the point $x_{st}$ corresponding to the standard pair to a point.

%One can calculate the tangent space to $X_{6,6}$ at the point corresponding to the standard pair and check that it has dimension %4.
Consider the point $x_0$ constructed in Lemma{\ref{x0}}. According to Theorem \ref{compon} it is a smooth point on a 4-dimensional component of $X(6,6)$.
Since formula (\ref{matrA}) for the transition matrix $A$ from the bases $\{p_i\}$ to the basis $\{q_i\}$  is a unitary matrix,
%the standard orthogonal pair
point $x_0$ is an element of ${\cal M}_{\mathbb R}$.
%It assigns to a representation of $B_{6,6}$ a new representation where every generating idempotent of $B_{6,6}$ is presented by %the Hermit conjugated projector, i.e. the involution acts on projectors by the rule:
%$$
%p\mapsto p^{\dag}.
%$$
%Let us consider the smooth part $W$ of the 4-dimensional component of $X(6,6)$, which exists by theorem \ref{compon}. Around a fixed point $x\in W$ of the involution, the action of the involution is complex-analytically isomorphic to the standard anti-holomorphic involution in a small ball around the zero point. Hence,
\end{proof}

{\bf Remark.} Since the transformation matrix from one mutually unbiased bases to the other one is known to be a complex Hadamard matrix, the above theorem implies existence of a 4 dimensional family of complex Hadamard matrices of size $6\times 6$.

\end{document}